\pdfoutput=1
\documentclass[11pt,reqno]{amsart}
\usepackage[paperheight=279mm,paperwidth=18cm,textheight=24cm,textwidth=14cm,includehead]{geometry}

\usepackage{amssymb, amsmath, amsthm, esint}
\usepackage{etoolbox}
\usepackage{mathtools}

\usepackage[unicode]{hyperref}
\hypersetup{
bookmarks=true,
colorlinks=true,
citecolor=[rgb]{0,0,0.5},
linkcolor=[rgb]{0,0,0.5},
urlcolor=[rgb]{0,0,0.75},
pdfpagemode=UseNone,
pdfstartview=FitH,
pdfdisplaydoctitle=true,
pdflang=en-US
}

\usepackage[style=alphabetic]{biblatex}
\title{A variable coefficient multi-frequency lemma}

\author[S.~Guo]{Shaoming Guo}
\address{Department of Mathematics, University of Wisconsin-Madison, Madison, WI, USA, and the IAS, Princeton, NJ, USA}
\email{shaomingguo@math.wisc.edu}

\author[P.~Zorin-Kranich]{Pavel Zorin-Kranich}
\address{Department of Mathematics, University of Bonn, Bonn, Germany}
\email{pzorin@uni-bonn.de}

\subjclass[2010]{42B15, 42B20, 42B25}

\def\R{\mathbb{R}}

\def\C{\mathbb{C}}
\def\Z{\mathbb{Z}}
\def\T{\mathbb{T}}
\def\Q{\mathbb{Q}}
\def\wt{\widetilde}

\def\mc{\mathcal}
\def\lesim{\lesssim}

\newcommand{\dif}{\mathrm{d}}
\newcommand{\one}{\mathbf{1}}

\newcommand{\supp}{\operatorname{supp}}

\theoremstyle{plain}
\newtheorem{thm}{Theorem}[section]

\newtheorem{lem}[thm]{Lemma}
\newtheorem{lemma}[thm]{Lemma}

\theoremstyle{remark}
\newtheorem{remark}[thm]{Remark}

\numberwithin{equation}{section}
\DeclarePairedDelimiter\abs{\lvert}{\rvert}

\DeclarePairedDelimiter\norm{\lVert}{\rVert}

\providecommand\given{}
\newcommand\SetSymbol[1][]{%
\nonscript\:#1\vert
\allowbreak
\nonscript\:
\mathopen{}}
\DeclarePairedDelimiterX\Set[1]\{\}{\renewcommand\given{\SetSymbol[\delimsize]}#1}
\DeclarePairedDelimiterXPP\EE[1]{\E}{\lparen}{\rparen}{}{\renewcommand\given{\SetSymbol[\delimsize]}#1} 
\DeclarePairedDelimiterX\inn[2]{\langle}{\rangle}{#1,#2}
\def\calT{\mathcal{T}}
\def\calF{\mathcal{F}}
\def\bfc{\mathbf{c}}
\def\calA{\mathcal{A}}
\def\calC{\mathcal{C}}

\begin{document}

\begin{abstract}
We show a variable coefficient version of Bourgain's multi-frequency lemma.
It can be used to obtain major arc estimates for a discrete Stein--Wainger type operator considered by Krause and Roos.
\end{abstract}
\maketitle


\section{Introduction}
Bourgain's multi-frequency lemma, first introduced in \cite{MR1019960}, allows one to estimate expressions of the type
\begin{equation}
\label{eq:Bourgain-multi-freq}
\norm[\Big]{ \sup_{t} \abs[\Big]{ \calF^{-1} \bigl( \sum_{\beta \in \Xi} S(\beta) \sigma_{t}(\cdot-\beta) \widehat{f} \bigr) } }_{2},
\end{equation}
where $\Xi$ is a $\delta$-separated set of frequencies and $(\sigma_{t})_{t}$ is a family of multipliers supported in a $\delta$-neighborhood of zero.
Expressions like \eqref{eq:Bourgain-multi-freq} arise when singular or averaging operators on $\Z^{n}$ are treated by the circle method.
The coefficients $S(\beta)$ are usually some type of complete exponential sums.

In this note, we address the problem of extending Bourgain's lemma to a setting in which the coefficients $S(\beta)$ in \eqref{eq:Bourgain-multi-freq} also depend on $t$.
This situation recently arose in \cite{arxiv:1803.09431,arXiv:1907.00405}.
Contrary to the classical case \eqref{eq:Bourgain-multi-freq}, the corresponding operator can no longer be easily represented as the composition of two Fourier multipliers.
We defer this application to Section~\ref{sec:application} and begin with the statement of our multi-frequency lemma.

Let $n\geq 1$ and let $\chi_{0}: \R^n\to [0,1]$ be a smooth bump function that is supported on $[-1, 1]^n$ and equals $1$ on $[-1/2, 1/2]^n$.
Let $A$ be a contractive invertible linear map on $\R^{n}$ and denote $\chi(\xi):=\chi_{0}(A^{-1}\xi)$, so that in particular $\chi$ equals $1$ on $U := A([-1/2,1/2]^{n})$.
Then $\phi = \calF_{\Z^n}^{-1} (\chi)$ is an $\ell^{1}$ normalized bump function, in the sense that $\norm{\phi}_{\ell^p(\Z^n)}\sim \abs{U}^{1-1/p}$.
In this article, we write $A \lesssim B$ if $A \leq CB$ with a constant $C$ depending only on the dimension $n$, unless indicated otherwise by a subscript.
We write $A \sim B$ if $A\lesssim B$ and $B \lesssim A$.

For a function $F$ from a totally odered set $\calT$ to a normed vector space $H$, we denote the \emph{$r$-variation seminorm} by
\[
\norm{F(t)}_{V^r_{t\in\calT}(H)}
=
\sup_{t_{0} \leq \dotsb \leq t_{J}} \bigl( \sum_{j=1}^{J} \abs{F(t_{j})-F(t_{j-1})}_{H}^{r} \bigr)^{1/r},
\]
where the supremum is taken over all finite increasing sequences in $\calT$.
The vector space $H$ may be omitted if it equals $\C$.
\begin{thm}\label{thm:var-coeff-multi-freq}
Let $\Xi$ be a finite set.
Let $g_{\beta} : \Z^{n} \to \mathbb{C}$, $\beta\in \Xi$, be functions such that, for every $x\in\Z^{n}$ and every sequence $(c_{\beta})_{\beta\in \Xi}$ of complex numbers, we have
\begin{equation}
\label{eq:3zz}
\norm[\big]{ \sum_{\beta\in\Xi} \phi(y) g_{\beta}(x+y) c_{\beta} }_{\ell^{2}_{y}}
\leq
A_{1} \abs{U}^{1/2} \norm{ c_{\beta }}_{\ell^{2}_{\beta}},
\end{equation}
for some $A_1>0$.
Let $\calT \subseteq \R$ be a finite set and let $(T_{t})_{t\in\calT}$ be a family of translation invariant operators on $\ell^{2}(\Z^{n})$ such that, for some $C_1<\infty$, some $\eta>0$ and every $r \in (2,3)$, we have
\begin{equation}
\label{eq:Vr-sigma-Znzz}
\norm{ \norm{ T_{t} f }_{V^{r}_{t\in\calT}} }_{\ell^2(\Z^n)}
\le C_1
(r-2)^{-\eta} \norm{f}_{\ell^2(\Z^n)}.
\end{equation}
Let $\Set{f_{\beta}}_{\beta\in \Xi} \subset \ell^{2}(\Z^{n})$ be functions with $\supp \hat f_{\beta} \subset U$ for every $\beta$.
Then, for any $q\in (2,\infty)$, we have
\begin{equation}
\label{eq:multi-freq-conclusion}
\norm{ \norm{ \sum_{\beta\in \Xi} g_{\beta}(x) (T_{t} f_{\beta})(x) }_{V^{q}_{t\in\calT}} }_{\ell^{2}_{x}(\Z^{n})}
\lesssim
\bigl( \frac{q(\log \abs{\Xi}+1)}{q-2} \bigr)^{\eta+1}
A_{1}
\norm{ \norm{ f_{\beta}}_{\ell^{2}(\Z^{n})} }_{\ell^{2}_{\beta\in \Xi}}.
\end{equation}
\end{thm}

\begin{remark}
The classical multi-frequency lemma corresponds to the case $g_{\beta}(x) = e^{2\pi i x \cdot \xi_{\beta}}$ with some $U$-separated frequencies $\xi_{\beta}$.
In this case, \eqref{eq:3zz} holds with $A_{1} \sim 1$.
\end{remark}

\begin{remark}
In \eqref{eq:multi-freq-conclusion}, one can use a different family $(T_{t,\beta})_{t\in\calT}$ for each $\beta\in\Xi$, as long as the bound \eqref{eq:Vr-sigma-Znzz} is uniform in $\Xi$.
\end{remark}

\section{Proof of the multi-frequency estimate}
\label{sec:proof}
The proof of Theorem~\ref{thm:var-coeff-multi-freq} is based on the arguments introduced in \cite{MR1019960} and further developed in \cite{MR2653686,MR3090139,arXiv:1402.1803,MR3280058}.
The first point where we deviate from the previous arguments is the following result, which extends \cite[Proposition 9.3]{MR3090139} and \cite[Lemma 3.2]{MR2653686}.

\begin{lemma}
\label{lem:pull-Vr-into-L2}
Let $\Xi$ be a finite set and $g_{\beta}$, $\beta\in\Xi$, complex-valued measurable functions on some measure space $Y$.
Assume that, for some $A_0<\infty$, we have
\begin{equation}
\label{eq:alm-orth-hypothesis}
\norm[\big]{ \sum_{\beta \in \Xi} g_{\beta}(y) c_{\beta} }_{L^{2}_{y}(Y)}
\leq
A_{0} \norm{c_{\beta}}_{\ell^{2}_{\beta}},
\end{equation}
for every sequence $(c_{\beta})_{\beta\in \Xi} \in \ell^{2}(\Xi)$.
Then, for every $2<r<q$, every countable totally ordered set $\calT$, and every collection of sequences $\Set{(c_{t,\beta})_{\beta\in\Xi}}_{t\in \calT} \subset \ell^{2}(\Xi)$, we have
\[
\norm[\big]{ \norm[\big]{ \sum_{\beta\in \Xi} g_{\beta}(y) c_{t,\beta} }_{V^{q}_{t}} }_{L^{2}_{y}(Y)}
\lesssim
\left( \frac{q}{q-r} + \frac{2}{r-2} \right)
A_{0} \abs{\Xi}^{\frac12 (\frac1{2} - \frac1r)/(\frac1{2} - \frac1q)}
\norm{ c_{t,\beta}}_{V^{r}_{t}(\ell^{2}_{\beta})},
\]
where the implicit constant is absolute.
\end{lemma}

The proof of Lemma~\ref{lem:pull-Vr-into-L2} relies on the following result.

\begin{lemma}[{\cite[Lemma 2.6]{MR3280058}}]
\label{lem:NOT:3.2}
Let $B$ be a normed space with norm $\norm{\cdot}_B$ and $B'$ the dual space of $B$. Let $Y$ be a measure space, and $g\in L^{p}(Y,B')$, $p\geq 1$.
Let also $\bfc = (c_{t})_{t\in\calT} \subset B$ with a countable totally ordered set $\calT$, and $q>p$.
Then
\[
\norm{ \norm{ \inn{c_{t}}{g(y)} }_{V^{q}_{t}} }_{L^{p}_{y}(Y)}
\lesssim
\int_{0}^{\infty} \min( M (J_{\lambda}(\bfc))^{\frac 1 p}, \norm{g}_{L^{p}(Y,B')} (J_{\lambda}(\bfc))^{\frac 1 q}) \dif\lambda,
\]
where $J_{\lambda}(\bfc)$ is the greedy jump counting function for the sequence $\bfc:=(c_{t})_{t\in \calT}$ at the scale $\lambda$,
\begin{equation}
\label{eq:1}
M = \sup_{c\in B, \norm{c}_B=1} \norm{ \inn{c}{g(y)} }_{L^{p}_{y}(Y)},
\end{equation}
and the implicit constant is absolute.
\end{lemma}
We will not need the definition of $J_{\lambda}(\bfc)$, only the fact that
\begin{equation}
\label{eq:jump-counting-fact}
J_{\lambda}(\bfc) \leq \norm{c_{t}}_{V^{r}_{t}(B)}^{r}/\lambda^{r}.
\end{equation}
\begin{proof}[Proof of Lemma~\ref{lem:pull-Vr-into-L2}]
We apply Lemma~\ref{lem:NOT:3.2} with $B=B'=\ell^{2}(\Xi)$, and $g(y)=(g_{\beta}(y))_{\beta\in\Xi}$.
By the hypothesis \eqref{eq:alm-orth-hypothesis}, we have $M \le A_{0}$, where $M$ was defined in \eqref{eq:1}.
Moreover,
\[
\norm{ g }_{L^{2}(Y, B')}
=
\bigl( \sum_{\beta\in\Xi} \norm{g_{\beta}}_{L^{2}}^{2} \bigr)^{1/2}
\leq
\abs{\Xi}^{1/2} A_{0},
\]
where we used \eqref{eq:alm-orth-hypothesis} with $(c_{\beta})$ being indicator functions of points.
Hence, we obtain
\[
\norm{ \norm{ \sum_{\beta\in \Xi} g_{\beta}(y) c_{t,\beta} }_{V^{q}_{t}} }_{L^{2}_{y}(Y)}
\lesssim
A_{0} \int_{0}^{\infty} \min(J_{\lambda}^{1/2}, \abs{\Xi}^{1/2} J_{\lambda}^{1/q}) \dif\lambda.
\]
Using \eqref{eq:jump-counting-fact} with $a:=\norm{ c_{t,\beta}}_{V^{r}_{t}(\ell^{2}_{\beta})}$ and splitting the integral at $\lambda_{0}=a \abs{\Xi}^{-1/(2r(1/2-1/q))}$, we obtain
\begin{multline*}
\int_{0}^{\lambda_{0}} \abs{\Xi}^{1/2} (a^{r}/\lambda^{r})^{1/q} \dif\lambda
+
\int_{\lambda_{0}}^{\infty} (a^{r}/\lambda^{r})^{1/2} \dif\lambda\\
=
\abs{\Xi}^{1/2}a^{r/q} (-r/q+1)^{-1} \lambda_{0}^{-r/q+1}
-
a^{r/2} (-r/2+1)^{-1} \lambda_{0}^{-r/2+1}\\
=
a \abs{\Xi}^{\frac12 (\frac1{2} - \frac1r)/(\frac1{2} - \frac1q)} ((1-r/q)^{-1} + (r/2-1)^{-1}).
\qedhere
\end{multline*}
\end{proof}

\begin{proof}[Proof of Theorem~\ref{thm:var-coeff-multi-freq}]
From the hypothesis \eqref{eq:Vr-sigma-Znzz} and Minkowski's inequality, it follows that
\begin{equation}\label{eq:cheap-vv-var-est-22zz}
\norm{ \norm{ T_{t} f_{\beta}(x) }_{V^{r}_{t}(\ell^{2}_{\beta\in \Xi})} }_{\ell^{2}_{x}(\Z^{n})}
\le C_1
\bigl( \frac{r}{r-2} \bigr)^{\eta} \norm{ \norm{ f_{\beta}}_{\ell^{2}(\Z^{n})} }_{\ell^{2}_{\beta\in \Xi}},
\end{equation}
initially for $r\in (2,3)$, but by monotonicity of the variation norms also for $r\in (2,\infty)$.

We use the Fourier uncertainty principle.
Let $R_{y}f(x) = f(x-y)$.
By the frequency support assumption on $f_{\beta}$, we have
\[
f_{\beta} = f_{\beta} * (\phi \tilde\phi),
\]
where $\tilde\phi$ is an $\ell^{\infty}$ normalized bump function with $\supp \widehat{\tilde\phi} \subseteq 4U$ such that $\widehat{\phi} * \widehat{\tilde\phi} \equiv 1$ on $U$.
It follows that
\begin{align*}
LHS\eqref{eq:multi-freq-conclusion}
&=
\norm{ \norm{ \sum_{\beta} g_{\beta}(x) T_{t}f_{\beta}(x) }_{V^{q}_{t}} }_{\ell^{2}_{x}}
\\ &=
\norm{ \norm{ \sum_{y\in \Z^{n}} (\phi \tilde\phi)(y) \sum_{\beta} g_{\beta}(x) (T_{t} R_{y} f_{\beta})(x)}_{V^{q}_{t}} }_{\ell^{2}_{x}}
\\ &\leq
\norm[\Big]{ \norm[\big]{ \tilde\phi(y)\cdot \norm{ \sum_{\beta\in \Xi} \phi(y) g_{\beta}(x) (T_{t} R_{y} f_{\beta})(x) }_{V^{q}_{t}} }_{\ell^{1}_{y}} }_{\ell^{2}_{x}}
\\ &\leq
\norm{\tilde{\phi}}_{\ell^{2}}
\norm{ \norm{ \norm{ \sum_{\beta\in \Xi} \phi(y) g_{\beta}(x) (T_{t}R_{y} f_{\beta})(x) }_{V^{q}_{t}} }_{\ell^{2}_{y}} }_{\ell^{2}_{x}}
\\ &\sim
\abs{U}^{-1/2}
\norm{ \norm{ \norm{ \sum_{\beta\in \Xi} \phi(y) g_{\beta}(x) (T_{t} f_{\beta})(x-y) }_{V^{q}_{t}} }_{\ell^{2}_{y}} }_{\ell^{2}_{x}}
\\ &=
\abs{U}^{-1/2}
\norm{ \norm{ \norm{ \sum_{\beta\in \Xi} \phi(y) g_{\beta}(x+y) (T_{t} f_{\beta})(x) }_{V^{q}_{t}} }_{\ell^{2}_{y}} }_{\ell^{2}_{x}}.
\end{align*}
For each fixed $x$, we will apply Lemma~\ref{lem:pull-Vr-into-L2} with the functions
\[
\tilde{g}_{\beta}(y) = \phi(y) g_{\beta}(x+y).
\]
By the hypothesis \eqref{eq:3zz}, the estimate \eqref{eq:alm-orth-hypothesis} holds with
\[
A_{0}
\leq
A_{1} \abs{U}^{1/2}.
\]
By Lemma~\ref{lem:pull-Vr-into-L2}, for any $2<r<q$, we obtain
\[
LHS\eqref{eq:multi-freq-conclusion}
\lesssim A_{1}
\left( \frac{q}{q-r} + \frac{2}{r-2} \right) \abs{\Xi}^{(\frac12 - \frac1r) \frac{q}{q-2}}
\norm{ \norm{ (T_{t}f_{\beta})(x) }_{V^{r}_{t}(\ell^{2}_{\beta\in \Xi})} }_{\ell^{2}_{x}}.
\]
By \eqref{eq:cheap-vv-var-est-22zz}, we obtain
\[
LHS\eqref{eq:multi-freq-conclusion}
\lesssim C_1 A_{1}
\left( \frac{q}{q-r} + \frac{2}{r-2} \right) \abs{\Xi}^{(\frac12 - \frac1r) \frac{q}{q-2}} \bigl( \frac{r}{r-2} \bigr)^{\eta}
\norm{ \norm{ f_{\beta} }_{\ell^{2}} }_{\ell^{2}_{\beta\in \Xi}}.
\]
Choosing $r$ such that $r-2 = (q-2)(\log \abs{\Xi} + 1)^{-1}$, this implies \eqref{eq:multi-freq-conclusion}.
\end{proof}

\section{An application}
\label{sec:application}

For a function $f: \Z^n\to \C$, consider the operator
\begin{equation}
\calC f(x) = \sup_{\lambda\in\R} \abs[\Big]{\sum_{y\in\Z^n\setminus\Set{0}} f(x-y) e(\lambda \abs{y}^{2d}) K(y)},\quad (x\in\Z^n),
\end{equation}
where $K$ is a Calderon-Zygmund kernel that satisfies the conditions as in \cite{arXiv:1907.00405} and $e(\lambda) = e^{2\pi i \lambda}$.
For instance, one can take $K$ to be the Riesz kernel.

Here we use Theorem~\ref{thm:var-coeff-multi-freq} to estimate the major arc operators arising in the proof of the $\ell^2$ bounds of $\calC$ in \cite{arXiv:1907.00405}.
We begin by recalling the approach, notation, and some results from \cite{arXiv:1907.00405}.

First, we apply a dyadic decomposition to $K$ and write
\begin{equation}
K=\sum_{j\ge 1} K_j,
\end{equation}
where $K_j:=K\cdot \psi_j$ and $\psi_j(\cdot )=\psi(2^{-j}\cdot)$ for some appropriately chosen non-negative smooth bump function $\psi$ which is compactly supported. Define the multiplier
\begin{equation}
m_{j, \lambda}(\xi):=\sum_{y\in \Z^n} e(\lambda \abs{y}^{2d}+\xi \cdot y)K_j(y).
\end{equation}
For a multiplier $m(\xi)$ defined on $\T^n$, we define
\begin{equation}
m(D) f(x):=
\int_{\T^n} m(\xi) \widehat{f}(\xi) e(x\cdot \xi) \dif\xi,
\quad x\in \Z^n.
\end{equation}
We also define the continuous version of the multiplier $m_{j, \lambda}$ by
\begin{equation}
\Phi_{j, \lambda}(\xi)
=
\int_{\R^n} e(\lambda \abs{y}^{2d}+\xi \cdot y)K_j(y) \dif y.
\end{equation}
The goal is to prove
\begin{equation}
\norm[\Big]{\sup_{\lambda\in \R} \abs{\sum_{j\ge 1} m_{j, \lambda}(D)f }}_{\ell^2}
\lesim
\norm{f}_{\ell^2}.
\end{equation}
Define the \emph{major arcs} (in the variable $\lambda$)
\begin{equation}
X_j=\bigcup_{\substack{a/q\in \Q, (a, q)=1\\ 1\le q\le 2^{\epsilon_1 j}}} \Set{\lambda\in \R: \abs{\lambda-a/q}\le 2^{-2d j+\epsilon_1 j}},
\end{equation}
where $\epsilon_1>0$ is a small fixed number that depends only on $d$.
The complement $\R\setminus X_j$ will be called a minor arc.

The contribution of the minor arcs was estimated in \cite{arxiv:1803.09431,arXiv:1907.00405} using a $TT^*$ argument in the spirit of \cite{MR1879821}, the result being that there exits $\gamma>0$ such that
\begin{equation}
\norm{ \sup_{\lambda\notin X_j} \abs{m_{j, \lambda}(D)f} }_{\ell^2(\Z)}
\lesim 2^{-j \gamma} \norm{f}_{\ell^2},
\end{equation}
holds for every $j\ge 1$.

On the major arcs in the variable $\lambda$, we have a good approximation of the discrete multiplier $m_{j, \lambda}$ by the continuous multiplier $\Phi_{j,\lambda}$.
For convenience, define
\begin{equation}
\Phi^*_{j, \lambda'}
=
\Phi_{j, \lambda'}\cdot \one_{\abs{\lambda'}\le 2^{-2dj+\epsilon_1 j}}.
\end{equation}
For an integer $1\le s\le \epsilon_1 j$, define
\begin{equation}
\mc{R}_s=\Set{ (a/q, {\bf b}/q)\in \Q\times \Q^n \given (a, {\bf b}, q)=1, q\in \Z\cap [2^{s-1}, 2^s) }.
\end{equation}
For $(\alpha, {\bf \beta})\in \mc{R}_s$, define a complete Gauss sum
\begin{equation}
S(\alpha, \beta)=q^{-n} \sum_{\substack{r=(r_1, \dots, r_n)\\ 0\le r_{1},\dotsc,r_{n}< q}} e(\alpha \abs{r}^{2d}+\beta\cdot r).
\end{equation}
Define $\chi_s(\cdot):=\chi_0(2^{10s}\cdot)$. Define
\begin{equation}
L^s_{j, \lambda}(\xi)
=\sum_{(\alpha, {\bf \beta})\in \mc{R}_s} S(\alpha, {\bf \beta}) \Phi^*_{j, \lambda-\alpha}(\xi-\beta) \chi_s(\xi-\beta).
\end{equation}
Define the error term
\begin{equation}
E_{j, \lambda}(\xi)
:=m_{j, \lambda}(\xi)\cdot \one_{X_j}(\lambda) -\Big(\sum_{1\le s\le \epsilon_1 j}L^s_{j, \lambda}(\xi) \Big).
\end{equation}
By a Sobolev embedding argument in the spirit of Krause and Lacey \cite{MR3658135} applied to the sup over $\lambda$, it was proved in \cite[Proposition 3.2]{arXiv:1907.00405} that there exists $\gamma>0$ such that
\begin{equation}
\norm{\sup_{\lambda\in X_j} \abs{E_{j, \lambda}(D)f}}_{\ell^2} \lesim 2^{-\gamma j}\norm{f}_{\ell^2}.
\end{equation}

It remains to bound the contribution from the multiplier
\begin{equation}
\sum_{j\ge 1} \sum_{1\le s\le \epsilon_1 j}L^s_{j, \lambda}(\xi)=\sum_{s\ge 1}\sum_{j\ge \epsilon_1^{-1}s} L^s_{j, \lambda}(\xi).
\end{equation}
To simplify notation, we introduce
\begin{equation}
L^s_{\lambda}=\sum_{j\ge \epsilon_1^{-1}s} L^s_{j, \lambda} \text{ and } \Phi^s_{\lambda}(\xi)=\sum_{j\ge \epsilon_1^{-1}s} \Phi^*_{j, \lambda}(\xi)\chi_s(\xi).
\end{equation}
By the triangle inequality applied to the sum over $s\ge 1$, it suffices to prove that there exists $\gamma>0$ such that
\begin{equation}\label{main_estimate_317}
\norm{\sup_{\lambda\in \R} \abs{L^s_{\lambda}(D)f}}_{\ell^2}
\lesim
2^{-\gamma s}\norm{f}_{\ell^2},
\end{equation}
for every $s\ge 1$.
This estimate is where our variable coefficient multi-frequency lemma, Theorem~\ref{thm:var-coeff-multi-freq}, will be useful.
The next two lemmas verify its assumptions \eqref{eq:3zz} and \eqref{eq:Vr-sigma-Znzz}, respectively.
Let
\begin{equation}
\begin{split}
& \mc{A}_s=\Set{\alpha\in \Q: (\alpha, \beta)\in \mc{R}_s \text{ for some } \beta},\\
& \mc{B}_s(\alpha)=\Set{\beta\in \Q^n: (\alpha, \beta)\in \mc{R}_s}.
\end{split}
\end{equation}
Moreover, define
\begin{equation}
L^{s, 2}_{\alpha}(\xi):=\sum_{\beta\in \mc{B}_s(\alpha)}S(\alpha, \beta) \chi_s(\xi-\beta).
\end{equation}
We have
\begin{lem}[{\cite[Proposition 3.3]{arXiv:1907.00405}}]\label{prop:weylsums}
There exists $\gamma>0$ depending on $d$ and $n$ such that
\begin{equation}\label{eqn:weylsums-main}
\norm{\sup_{\alpha\in \mc{A}_s} \abs{L^{s, 2}_{\alpha}(D)f}}_{\ell^2} \lesim 2^{-\gamma s}\norm{f}_{\ell^2},
\end{equation}
for every $s\ge 1$.
\end{lem}

\begin{lem}\label{variation_lemma_GRY}
For every $r\in (2, 3)$, we have
\begin{equation}
\norm{ \norm{ \Phi^s_{\lambda}(D)f }_{V^{r}_{\lambda\in (0, 1]}} }_{\ell^2}
\lesim_{d, n}
(r-2)^{-1} \norm{f}_{\ell^2}.
\end{equation}
\end{lem}
\begin{proof}[Proof of Lemma~\ref{variation_lemma_GRY}.]
By the transference principle of Magyar, Stein, and Wainger in \cite[Proposition 2.1]{MR1888798}, it suffices to prove that
\begin{equation}
\norm{ \norm{ \Phi^s_{\lambda}(D)f }_{V^{r}_{\lambda\in (0, 1]}}}_{L^2(\R^n)}
\lesim_{d, n}
(r-2)^{-1} \norm{f}_{L^2(\R^n)},
\end{equation}
with constants independent of $s$. This was essentially established in Guo, Roos and Yung \cite{arXiv:1710.10988}, with minor changes detailed in Roos \cite[Section 7]{arXiv:1907.00405}.
\end{proof}

Now we are ready to prove \eqref{main_estimate_317}. We linearize the supremum and aim to prove
\begin{equation}\label{ee200603e1.6}
\norm[\big]{L^s_{\lambda(x)}(D)f(x)}_{\ell^2_x} \lesim 2^{-\gamma s} \norm{f}_{\ell^2},
\end{equation}
where $\lambda: \Z^n\to (0, 1]$ is an arbitrary function.
For each $x\in\Z^n$, $\alpha(x)$ is defined as the unique $\alpha\in\calA_s$ such that $\abs{\lambda(x)-\alpha}\le 2^{-3s}$ (say), or as an arbitrary value from the complement of $\calA_s$ if no such $\alpha$ exists (in this case, $L^s_{\lambda(x)}(\xi)=0$).
By definition, the term we need to bound in \eqref{ee200603e1.6} can be written as
\begin{equation}\label{ee200603e1.8}
\begin{split}
&\sum_{\beta\in \mc{B}_s(\alpha(x))} \int S(\alpha(x), \beta) \Phi^s_{\lambda(x)-\alpha(x)}(\xi-\beta) \hat{F}_{\beta}(\xi) e(\xi x) d\xi,
\end{split}
\end{equation}
where
\begin{eqnarray}
\hat{F}_{\beta}(\xi)=\hat{f}(\xi) \wt{\chi}_s(\xi-\beta),
\end{eqnarray}
and $\wt{\chi}_s(\cdot)=\wt{\chi}_0(2^{10s}\cdot)$ for some appropriately chosen compactly supported smooth bump function $\wt{\chi}_0$ with $\wt{\chi}_{0}\chi_{0}=\chi_{0}$.
We apply Theorem~\ref{thm:var-coeff-multi-freq} with
\begin{equation}
\Xi=\Set{{\bf b}/q: {\bf b}\in \Z^n, q\in \Z\cap [2^{s-1}, 2^s)},
\end{equation}
$t=\lambda$, $\phi=\mc{F}^{-1}_{\Z^n}(\chi_s)$, $U$ the support of $\chi_s$, and $T_t=\Phi^s_{t}(D)=\Phi^s_{\lambda}(D)$,
and any fixed $q$, say, $q=3$.
The hypothesis \eqref{eq:Vr-sigma-Znzz} with $\eta=1$ is then given by Lemma~\ref{variation_lemma_GRY}.
In \eqref{eq:multi-freq-conclusion}, we take
\begin{equation}
f_{\beta}(y)=F_{\beta}(y) e(-\beta y),
\end{equation}
and
\begin{equation}\label{20200628e3_20}
g_{\beta}(x)=\one_{\beta\in \mc{B}_s(\alpha(x))}\cdot S(\alpha(x), \beta) e(\beta x).
\end{equation}
Since $\Phi_{1}^s=0$ for all $s$, the $V^{q}$ norm on the left-hand side of \eqref{eq:multi-freq-conclusion} controls the supremum over $\lambda$.
We apply Theorem~\ref{thm:var-coeff-multi-freq} and bound term \eqref{ee200603e1.8} by
\begin{equation}
s^2 A_1 \norm{f}_{L^2},
\end{equation}
where $A_1$ is the constant in \eqref{eq:3zz} under the above choice of $g_{\beta}$. It remains to prove that
\begin{equation}\label{200603e1.14}
A_1 \lesim 2^{-\gamma s} \text{ for some } \gamma>0.
\end{equation}
To do so, we will apply Lemma~\ref{prop:weylsums}.

Regarding the left hand side of \eqref{eq:3zz}, we apply a change of variable and write it as
\begin{equation}
\norm[\Big]{ \sum_{\beta\in\Xi} \phi(y-x) g_{\beta}(y) c_{\beta} }_{\ell^{2}_{y}}
=
\norm[\Big]{\sum_{\beta\in \mc{B}_s(\alpha(y))} \phi(y-x) S(\alpha(y), \beta) e(\beta y) c_{\beta}}_{\ell^2_y}.
\end{equation}
We write a linearization of the left hand side of \eqref{eqn:weylsums-main} as
\begin{equation}
\begin{split}
& \sum_{\beta\in \mc{B}_s(\alpha(y))} \int S(\alpha(y), \beta) \hat{F}_{\beta}(\xi) e(\xi y) d\xi= \sum_{\beta\in \mc{B}_s(\alpha(y))} S(\alpha(y), \beta) F_{\beta}(y),
\end{split}
\end{equation}
where
\begin{equation}\label{200603e1.17}
\hat{F}_{\beta}(\xi)=\chi_s(\xi-\beta) \hat{f}(\xi).
\end{equation}
In the end, we just need to pick
\begin{equation}
\hat{f}(\xi)=
\sum_{\beta\in \mc{B}_s^{\sharp}}c_{\beta}\cdot \tilde{\chi}_s(\xi-\beta) e(x(\beta-\xi)),
\end{equation}
The desired estimate \eqref{200603e1.14} follows as
\begin{equation}
\norm{f}_{\ell^2} \sim \abs{U}^{1/2} \norm{c_{\beta}}_{\ell^2_{\beta}}.
\end{equation}
This finishes the proof of \eqref{200603e1.14}, thus the proof of the desired estimate \eqref{ee200603e1.6}.

\printbibliography
\end{document}